\documentclass[11pt]{amsart}
\usepackage[T1]{fontenc}
\usepackage[latin1]{inputenc}
\usepackage{amsmath}
\usepackage{amssymb}
\usepackage{amsthm}
\usepackage{dsfont}
\usepackage[a4paper,left=2.5cm,right=2.5cm,top=3cm, bottom=3cm]{geometry}
\usepackage{color}
\usepackage{graphics}
\usepackage{tikz}
\usetikzlibrary{patterns}
\usepackage{caption}
\usepackage{subfigure}
\usepackage{hyperref}
\hypersetup{
	colorlinks,
	linkcolor={lblue},
	citecolor={lblue},
	urlcolor={orange}
}

\newtheorem{satz}{Proposition}[section]
\newtheorem{lem}[satz]{Lemma} 
\newtheorem{remark}[satz]{Remark}
\newtheorem{thm}[satz]{Theorem}

\numberwithin{equation}{section}

\definecolor{gray}{gray}{0.50}
\definecolor{lred}{rgb}{1.0,0.5,0.5}
\definecolor{lblue}{rgb}{0.0, 0.313, 0.608}
\definecolor{dgreen}{rgb}{0,1,1}
\definecolor{luh-dark-blue}{rgb}{0.0, 0.313, 0.608}

\newcommand{\chookrightarrow}{\mathrel{\lhook\joinrel\relbar\kern-.8ex\joinrel\lhook\joinrel\rightarrow}}

\newcommand{\R}{\mathbb{R}}

\newcommand{\N}{\mathbb{N}}		   
				
\newcommand{\C}{\mathbb{C}}	

\newcommand{\e}{\varepsilon}

\DeclareMathOperator{\re}{Re}

\title[Well-posedness and stability for a thin film equation with  surfactant]{Well-posedness and stability for a mixed order system arising in thin film equations with surfactant}
\author{Gabriele Bruell}
\address{Institute for Analysis, Karlsruher Institute of Technology (KIT), D-76128 Karlsruhe, Germany}
\email{gabriele.bruell@kit.edu}
\thanks{Date: \today }

\begin{document}
	
\maketitle

\begin{abstract}
	The objective of the present work is to provide a well-posedness result for a capillary driven thin film equation with insoluble surfactant. The resulting parabolic system of evolution equations is not only strongly coupled and degenerated, but also of mixed orders. To the best of our knowledge the only well-posedness result for a capillary driven thin film with surfactant is provided in \cite{B1} by the same author, where a severe smallness condition on the surfactant concentration is assumed to prove the result. Thus, in spite of an intensive analytical study of  thin film equations with surfactant during the last decade, a proper well-posedness result is still missing in the literature. It is the aim of the present paper to fill this gap.  \\
	Furthermore, we apply a recently established result on asymptotic stability in interpolation spaces \cite{MW} to prove that the flat equilibrium of our system is asymptotically stable.
\end{abstract}

\bigskip

\section{Introduction}
The present paper is a note on a so far missing piece in the analysis of a thin film equation with insoluble surfactant. Classically, the evolution equations for a thin fluid film equipped with a layer of insoluble surfactant was derived from the Navier--Stokes equations with an advection-diffusion equation on the free surface by Jensen \& Grotberg \cite{JG}  using lubrication approximation and cross-sectional averaging. The thin fluid film is assumed to be uniform in one horizontal direction and the contact angle between the fluid and the impermeable flat bottom is zero, which corresponds to the frame of so-called \emph{complete wetting}. 

\begin{figure}[h!]
	\centering
	\begin{tikzpicture}[domain=0:3*pi, scale=1] 
	\draw[color=black] plot (\x,{0.3*cos(\x r)+1}); 
	\draw[very thick, smooth, variable=\x, luh-dark-blue] plot (\x,{0.3*cos(\x r)+1}); 
	\fill[luh-dark-blue!10] plot[domain=0:3*pi] (\x,0) -- plot[domain=3*pi:0] (\x,{0.3*cos(\x r)+1});
	\draw[very thick,<->] (3*pi+0.4,0) node[right] {$x$} -- (0,0) -- (0,2) node[above] {$z$};
	\draw[very thick,->] (2*pi,0) -- (2*pi,1.3);
	\node[right] at (2*pi,0.5) {$h(t,x)$};
	\coordinate[label=above:{$\Gamma(t,x)$}] (A) at (pi,0.72);
	\fill[color=luh-dark-blue] (A) circle (1.5pt);
	\draw[-] (0,-0.3) -- (0.3, 0);
	\draw[-] (0.5,-0.3) -- +(0.3, 0.3);
	\draw[-] (1,-0.3) -- +(0.3, 0.3);
	\draw[-] (1.5,-0.3) -- +(0.3, 0.3);
	\draw[-] (2,-0.3) -- +(0.3, 0.3);
	\draw[-] (2.5,-0.3) -- +(0.3, 0.3);
	\draw[-] (3,-0.3) -- +(0.3, 0.3);
	\draw[-] (3.5,-0.3) -- +(0.3, 0.3);
	\draw[-] (4,-0.3) -- +(0.3, 0.3);
	\draw[-] (4.5,-0.3) -- +(0.3, 0.3);
	\draw[-] (5,-0.3) -- +(0.3, 0.3);
	\draw[-] (5.5,-0.3) -- +(0.3, 0.3);
	\draw[-] (6,-0.3) -- +(0.3, 0.3);
	\draw[-] (6.5,-0.3) -- +(0.3, 0.3);
	\draw[-] (7,-0.3) -- +(0.3, 0.3);
	\draw[-] (7.5,-0.3) -- +(0.3, 0.3);
	\draw[-] (8,-0.3) -- +(0.3, 0.3);
	\draw[-] (8.5,-0.3) -- +(0.3, 0.3);
	\draw[-] (9,-0.3) -- +(0.3, 0.3);
	\end{tikzpicture} 
	\caption{Scheme of a  thin film flow with insoluble surfactant}\label{Fig1}
\end{figure}
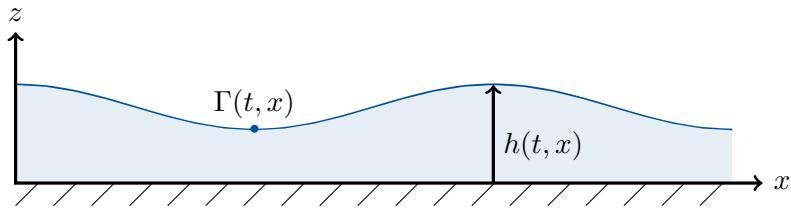

Set $\Omega:=(0,L)$ and denote by $h=h(t,x)$ and $\Gamma=\Gamma(t,x)$ the film height and the surfactant concentration at time $t\in [0,\infty)$ and space $x\in\Omega$.
Under consideration of different driving forces, such as capillarity or gravitation, the resulting system of evolution equations for the fluid height $h$ and the surfactant concentration $\Gamma$ consist either of two parabolic, strongly coupled and degenerated equations of second order in the case when gravitational forces dominate or the evolution equations constitute a system of mixed orders  if capillary effects are taken into account -- a fourth-order equation for the evolution of the fluid height coupled with a second-order equation for the surfactant concentration.
 Considering capillary effects as the only driving force, the evolution of a thin film endowed with a layer of insoluble surfactant on the surface is given by
\begin{equation}
\label{eq:system}
	\left\{ 
		\begin{array}{lcl}
			\partial_t h + \partial_x \left[\frac{1}{3}h^3\partial_x^3 h + \frac{1}{2}h^2\partial_x \sigma(\Gamma) \right]=0,\\[5pt]
			\partial_t \Gamma + \partial_x\left[\frac{1}{2}h^2\Gamma\partial_x^3h + h\Gamma \partial_x \sigma(\Gamma)\right]=\mathcal{D}\partial_x^2 \Gamma,
		\end{array}
	\right.
\end{equation}
in $(0,\infty)\times \Omega$. The system is supplemented with the following no-flux boundary conditions
\begin{equation}
\label{eq:boundary}
	\partial_x h = \partial_x^3 h =0 \qquad \mbox{and}\qquad \partial_x \Gamma =0 \qquad \mbox{on} \quad \partial \Omega =\{0,L\},
\end{equation}
and initial conditions
\begin{equation}
\label{eq:initial}
	h(0,x)=h_0(x), \qquad \Gamma(0,x)=\Gamma_0(x).
\end{equation}

The constant $\mathcal{D}>0$ appearing on the right hand side of the transport equation for the surfactant concentration denotes a surface diffusion coefficient and the function $\sigma$ represents the surface tension coefficient which is (decreasingly) dependent on the surfactant concentration. The function $\sigma$ is given and we assume throughout our analysis that it satisfies
\[
\sigma\in C^2(\R)\qquad\mbox{and}\qquad -\sigma^\prime(s)\geq 0\quad\mbox{for all}\quad s\geq 0.
\]

The corresponding system for a thin film evolution driven by gravitational forces only can be recovered from \eqref{eq:system} by replacing the appearing third-order operators $\partial_x^3$ with $-\partial_x$. For the resulting second-order system a rigorous mathematical analysis is provided by \cite{EHLW11, EHLW1}. In \cite{EHLW1}  the authors derive the mathematical model in the case of soluble surfactant by means of lubrication approximation and prove  local well-posedness as well as asymptotic stability of the (one and only) flat equilibrium.  In view of the degeneracy of the equations with respect to the film height, it is not expected that classical solutions exist globally in time. Based on an associated energy functional, which provides sufficient a priori estimates, the existence of nonnegative global weak solutions is investigated in \cite{EHLW11}. Concerning the fourth-order counter part, that is system \eqref{eq:system}, when capillary effects instead of gravitation form the driving force, the existence of nonnegative global weak solutions is studied in \cite{CT13, EHLW4th, GW06}. Moreover, the existence and asymptotic behavior of global weak solutions of a thin film equation with insoluble surfactant under the influence of gravitational, capillary as well as van der Waals forces (the system of evolution equations derived by Jensen \& Grotberg \cite{JG}) is subject of \cite{BG2}. Eventually,  a corresponding analysis concerning modeling, well-posedness, asymptotic stability of equilibria, and weak solutions is carried out in \cite{B1,B2} for a \emph{two-phase} thin film equation with insoluble surfactant under consideration of capillary effects.  We would like to point out that the proof of the well-posedness result presented in \cite{B1} for the capillary driven film, which also implies the well-posedness of \eqref{eq:system}, is restricted to a  smallness assumption on the initial data, which is not desirable. 

\medskip

In view of this unsatisfactory condition, a proper well-posedness result for the capillary driven thin film equation \eqref{eq:system} with insoluble surfactant is still missing in the literature. It is the aim of the present paper to close this gap; thereby also providing a basis for proving well-posedness of comparable systems arising in thin film equations with surfactant.
System \eqref{eq:system} can be rewritten as a quasilinear evolution equation in a suitable positive cone of Sobolev spaces  to be made precise below:
\[
u_t-A(u)u(t)=F(u(t)),\quad t>0,\qquad u(0,x)=u_0(x),
\]
where $u=(h,\Gamma)$. Here $A(u)$ is the leading order matrix having the form
\[
A(u):= \begin{pmatrix} -a_{11}(u)\partial_x^4 & a_{12}(u)\partial_x^2 \\  -a_{21}(u)\partial_x^4 & a_{22}(u)\partial_x^2 \end{pmatrix}
\]
with coefficients $a_{ij}(u)>0$, $1\leq i,j\leq 2$. The function $F$ comprises lower-order terms. The approach implemented in \cite{B1} relies on the fact that the operators  on the diagonal of $A$ generate analytic semigroups on suitable domains. The well-posedness result is then a consequence of \cite[Theorem I.1.6.1]{Amann95} on matrix generators, provided the off diagonal terms satisfy a certain relation, which can be forced by a smallness assumption on the surfactant concentration in higher-order Sobolev spaces. Concerning the present work and its aim  to provide a well-posedness result of \eqref{eq:system} without any smallness assumptions on the initial data, we proceed more directly and prove resolvent estimates for the mixed-order matrix operator $A$ by means of parameter ellipticity in the sense of Douglis--Nirenberg. Verifying a certain Loptatinskij--Shapiro condition (cf. (C2) in Section \ref{S:P}), we affirm that the mixed-order matrix $A$ with fixed coefficients in a certain regularity space generates an analytic semigroup. The well-posedness result then follows from a classical result on abstract parabolic equations by Amann in \cite[Section 12]{AmannNon} and \cite[Theorem 1.1]{MW}.
\medskip

Eventually, we recall that the one and only equilibrium of system \eqref{eq:system} is given by the flat state, which is uniquely determined by the initial data. We apply a recently established result in \cite{MW} to prove that the flat equilibrium is asymptotically stable in interpolation spaces. This improves the existing stability result in \cite[Theorem 4.5]{B1}.

\medskip

The paper is structured as follows: In Section \ref{S:Q} we introduce the needed Sobolev spaces, rewrite \eqref{eq:system} as a quasilinear evolution equation and state our main result Theorem \ref{T1}.The remainder of the section is devoted to the proof of the theorem. We introduce the notion of parameter ellipticity in the sense of Douglis--Nirenberg and the Lopatinskij--Shapiro condition. Eventually, we verify that the mixed-order matrix $A(\bar u)$ for fixed $\bar u$ satisfies the above condition, which provide sufficient resolvent estimates to guarantee that $A(\bar u)$ generates an analytic semigroup. Section \ref{S:A} is concerned with the asymptotic stability result, which is based on \cite[Theorem 1.3]{MW}. In our case the claim  follows immediately provided that $A(u_*)$, where $u_*$ is the equilibrium solution,  has negative spectral bound on the subset of zero mean functions.

\medskip

We close the introduction with some comments on notation: If $X,Y$ are Banach spaces, then the set of all linear and bounded operators from $X$ to $Y$ is denoted by $\mathcal{L}(X,Y)$. If $X=Y$, we use the abbreviation $\mathcal{L}(X):= \mathcal{L}(X,X)$.  We write $A\in \mathcal{H}(X,Y)$ when $A$ is a linear, unbounded operator on $Y$ with domain $X$, which generates an  analytic semigroup on $\mathcal{L}(Y)$. The notation  $c=c(p_1,p_2,\ldots)>0$ is used, whenever we want to emphasize  that the constant $c>0$ depends on the parameters $p_1,p_2,\ldots$.
Furthermore, we denote by $\N_0:=\N\cup \{0\}$ the set of natural numbers including zero. 

\bigskip

\section{Well-posedness}
\label{S:Q}
We declare suitable Banach spaces on which the system of evolution equations \eqref{eq:system} will be studied.  In what follows we set $L_2:=L_2(\Omega,\R)$ and denote the $L_2$-based Bessel potential spaces on $\Omega=(0,L)$ with values in $\R$ of order $s>0$ by $H^s$. We aim to rewrite  \eqref{eq:system}--\eqref{eq:initial} in a setting appropriate to apply abstract parabolic theory to obtain our well-posedness result. For this purpose we define the following spaces, which incorporate the boundary conditions \eqref{eq:boundary} as soon as sufficient regularity is available. For any $s>0$, we set
$$ H_B^{s}:= \left\{f\in H^{s}\mid \partial_x^{2l+1}f=0  \:\mbox{at}\: x=0,L,  \:\mbox{for all}\: l\in\N_0 \:\mbox{with}\: 2l+1< s-\frac{1}{2}\right\}.$$
The space $H^4_B\times H^2_B$ plays a natural role in our analysis. For $\theta \in (0,1)\setminus\{\frac{3}{8},\frac{7}{8}\}$ the complex interpolation spaces between $L_2\times L_2$ and $H^{4}_B\times H^{2}_B$ are given by
\[
[L_2\times L_2, H^{4}_B\times H^{2}_B]_\theta=H^{4\theta}_B\times H^{2\theta}_B,
\]
cf. e.g. \cite[Theorem 4.3.3]{TriI}. Notice that for any $\theta>\frac{7}{8}$, the complex interpolation space above includes the boundary conditions \eqref{eq:boundary}.
 In view of the degeneracy of system \eqref{eq:system}, we require positivity of  classical solutions $u=(h,\Gamma)$ and set
 \begin{equation}\label{eq:O}
 \mathcal O:=\{u=(h, \Gamma)\in H^3\times H^1\mid h>0\mbox{ and } \Gamma>0 \mbox{ on } \overline \Omega\},
 \end{equation}
 which is a nonempty open subset of $(H^3\times H^1)\cap C(\Omega,(0,\infty)^2)$. 
For each $u=(h,\Gamma)\in \mathcal O$ we introduce the leading order matrix operator
\[
	A:\mathcal O \to \mathcal{L}(H^4_B\times H^2_B, L_2\times L_2)
\]
by
\begin{equation}\label{eq:A}
A(u):=-\begin{pmatrix}\frac{h^3}{3}\partial_x^4 & \frac{h^2}{2}\sigma^\prime(\Gamma)\partial_x^2  \\
						\frac{h^2}{2}\Gamma \partial_x^4  & (h\Gamma\sigma^\prime(\Gamma)-\mathcal{D})\partial_x^2 \end{pmatrix}.
\end{equation}
Clearly, for each fixed $u\in \mathcal O$, the operator $A(u)$ is a linear operator acting on $H^4_B\times H^2_B$. The lower order terms are comprised in the function $F:\mathcal O \to L_2\times L_2$, given by
\[
	F(u):= -\begin{pmatrix}
			h^2\partial_x h\partial_x^3h & -\partial_x\left(\frac{h^2}{2}\sigma^\prime(\Gamma) \right)\partial_x \Gamma \\
			\partial_x \left(\frac{h^2}{2}\Gamma\right)\partial_x^3h & -\partial_x\left(h\Gamma\sigma^\prime(\Gamma)\right)\partial_x\Gamma
		\end{pmatrix}.
\]
 With this notation, the system \eqref{eq:system} can be expressed as an autonomous quasilinear evolution equation in $L_2\times L_2$:
\begin{equation}\label{eq:equation}
	\dot{u}(t) -A(u(t))u(t)=F(u(t)), \quad t>0, \qquad u(0)=u_0,
\end{equation}
where $u_0=(h_0,\Gamma_0)$ is the initial datum. 
 Our main result is concerned with the local well-posedness of \eqref{eq:equation} and reads as follows:
\begin{thm}[Local well-posedness]\label{T1} 
Let $\alpha\in \left(\frac{7}{8},1\right)$ and $u_0=(h_0,\Gamma_0)\in \mathcal O_\alpha:= \mathcal O \cap (H^{4\alpha}_B\times H^{2\alpha}_B)$. Then, the problem \eqref{eq:equation} possesses a unique maximal strong solution 
\[
	u=(h,\Gamma)\in C^1((0,T(u_0)),L_2\times L_2)\cap C((0,T(u_0),H^4_B\times H^2_B)\cap C([0,T(u_0)),\mathcal O_\alpha),
\]
where $T(u_0)>0$  is the maximal time of existence depending on the initial datum. Moreover, the solution depends  continuously on the initial datum.
\end{thm}

In order to prove the above theorem, we use the following abstract existence and uniqueness result for abstract quasilinear problems, which can be found in \cite[Section 12]{AmannNon}, \cite[Theorem 1.1]{MW}:

\begin{thm}
	\label{T2}
	Let $(E_0,E_1)$ be a densely injected Banach couple and for $\theta\in (0,1)$, let $ E_\theta:=[E_0, E_1]_\theta$ be the complex interpolation space between $E_0$ and $E_1$. Assume that  $\mathcal O\subset E_\theta$ for some $\theta\in (0,1)$ and set $\mathcal O_\delta:= \mathcal O \cap E_\delta$ for $\delta\in (0,1)$. Suppose that $0<\gamma\leq \beta<\alpha<1$, that $\mathcal O_\beta \subset E_\beta$ is an open, nonempty subset, and that
	\[
		(A,F)\in C^{1-}(\mathcal O_\beta,\mathcal{H}(E_1,E_0)\times E_\gamma).
	\]
	Then, 
	\begin{itemize}
	\item[i)] \textbf{Existence:} the quasilinear equation
\begin{equation}\label{eq:equationA}
	\dot{u}(t) -A(u(t))u(t)=F(u(t)), \quad t>0, \qquad u(0)=u_0,
\end{equation}
	possesses for each $u_0\in\mathcal O_\alpha$ a maximal strong solution $u$ having the regularity
	\[
	u\in C^1((0,T(u_0)),E_0)\cap C((0,T(u_0)),E_1)\cap C([0,T(u_0)), \mathcal O_\alpha),
	\]
	where $T(u_0)\in (0,\infty]$ is the maximal time of existence.
	If $u_0\in E_1$, then 
	\[u\in C^1((0,T(u_0)), E_0)\cap C([0,T(u_0)), E_1).
	\]
	\item[ii)] \textbf{Uniqueness:} If
	\[
	\tilde u\in C^1((0,T],E_0)\cap C((0,T],E_1)\cap C^\eta([0,T], \mathcal O_\beta)
	\]
	is a strong solution of \eqref{eq:equationA} for some $T>0$ and $\eta \in (0,1)$, then $\tilde u=u$ on $[0,T]$.
	Moreover, the solution depends continuously on the initial datum.
	\end{itemize}
\end{thm}

\medskip

Let us identify
\[
E_0:=L_2\times L_2\qquad \mbox{and}\qquad E_1:= H^4_B\times H^2_B.
\]
 Then $E_1 \hookrightarrow E_0$ is a densely injected Banach couple.  
For any $\theta\in (0,1)\setminus\{\frac{3}{8},\frac{7}{8}\}$, the interpolation space $E_\theta=[E_0,E_1]_\theta$ is given by $E_\theta= H^{4\theta}_B\times H^{2\theta}_B$ . Let $\alpha\in (\frac{7}{8},1)$ and choose $\beta\in (\frac{7}{8},\alpha)$. Setting $\gamma=\frac{4\beta-3}{4}$, we obtain that $\gamma\in (\frac{1}{8},\frac{1}{4})$ and
\[
0<\gamma<\beta<\alpha<1.
\]
By Sobolev embedding, we have that $E_\beta =H^{4\beta}_B\times H^{2\beta}_B \subset C^3(\overline \Omega, \R)\times C^1(\overline\Omega, \R)$; whence
\[
A: \mathcal O_\beta \to \mathcal{L}(E_1,E_0),
\]
where $O_\beta =\mathcal O\cap E_\beta$ and $\mathcal O$ is the open subset defined in \eqref{eq:O}.
Moreover, we have that $E_\gamma= H^{4\beta-3}_B\times H^{\frac{4\beta-3}{2}}_B$ and
\[
F:\mathcal O_\beta \to H^{4\beta-3}_B\times H^{2\beta-1}_B\subset E_\gamma.
\]
In view of the smooth dependence of $A$ and $F$ with respect to their coefficients, it is clear that
\[
(A,F)\in C^{\infty}(\mathcal O_\beta, \mathcal L(E_1,E_0)\times E_\gamma)
\]
and
 our main result Theorem \ref{T1} is a direct consequence of Theorem \ref{T2}, provided we prove that 
 \begin{itemize}
 \item[i)] for any $\bar u\in \mathcal O_\beta$ the operator $A(\bar u)$ with domain $H^4_B\times H^2_B$ is the generator of an analytic semigroup in $\mathcal{L}(L_2\times L_2)$;
 \item[ii)] any strong solution
 \[
 \tilde u\in C^1((0,\tilde T),L_2\times L_2)\cap C((0,\tilde T),H^4_B\times H^2_B)\cap C([0,\tilde T), \mathcal O_\alpha)\qquad \tilde T\in (0,\infty]
 \]
 satisfies
 \[
 \tilde u \in  C^\eta([0,T],  H^{4\beta}_B\times H^{2\beta}_B) \quad \mbox{for all}\quad T\in (0,\tilde T),
 \]
 for some $\eta \in (0,1)$.
\end{itemize}

\medskip

\subsection{Proof of Theorem \ref{T1}}
\label{S:P}
	Let $\alpha\in (\frac{7}{8},1)$ and $\beta \in (\frac{7}{8},\alpha)$. Given an initial datum $u_0\in \mathcal O_\alpha$, let us first prove the existence of a maximal strong solution of \eqref{eq:equation}. Subsequently we show that any such solution is unique with respect to the initial datum.
	
	\subsubsection{Existence}
	Let $\mathcal { O_\beta}=\mathcal O \cap H^{4\beta}_B\times H^{2\beta}_B$, where $\mathcal O$ is as defined in \eqref{eq:O} and fix $\bar u \in \mathcal{O_\beta}$.
	In view of Theorem \ref{T2} i) the existence of a maximal solution is guaranteed by verifying that for any $\bar u\in \mathcal O_\beta$ the operator $A(\bar u)$ with domain $H^4_B\times H^2_B$ is the generator of an analytic semigroup in $\mathcal{L}(L_2\times L_2)$.
 This is achieved by first smoothening the coefficients and considering the operator $A(\bar u_\e)$, where $\bar u_\e \in (C^\infty(\overline \Omega))^2$ such that 
\begin{equation}\label{eq:app}
\|\bar u-\bar u_\e\|_{\infty}<\e.
\end{equation}
 We prove resolvent estimates for $A(\bar u_\e)$ in $L_2\times L_2$ using the theory of \emph{parameter elliptic Douglis--Nirenberg systems} and infer that $A(\bar u_\e)$ is the generator of an analytic semigroup.
By \eqref{eq:app} it is ensured that there exists a constant $c>0$ such that
\[
	\|A(\bar u_\e)-A(\bar u)\|_{\mathcal{L}(H^4_B\times H^2_B,L_2\times L_2)}< c\e .
	\]
As a consequence of perturbation arguments for analytic semigroups (cf. e.g. \cite[Theorem~3.2.1]{Pazy}), one obtains that $A(\bar u)$ itself is the generator of an analytic semigroup.

\medskip
Let us start by collecting some definitions and implications concerning systems of elliptic boundary value problems which can also be found in more detail and generality in e.g. \cite[Chapter 6]{AES}.
Let $\Omega\subset \R$ be a bounded domain and $A$ a matrix operator defined as
\[
	A=\begin{pmatrix} A_{11} & A_{12} \\ A_{21} & A_{22} \end{pmatrix},
\]
where $A_{ij}:=a_{ij}D^{\alpha_{ij}}$ with
$a_{ij}\in C^\infty(\overline{\Omega})$ and $\alpha_{ij}=\mbox{ord} A_{ij}\in \N_0$ for $1\leq i,j\leq 2. $
Let $\{l_1,l_2, m_1,m_2\}$ be a set of integers and $\alpha_{ij}\leq l_i+m_j$ ($ A_{ij}=0$ if $l_i+m_j<0$). Similarly, we define a matrix of boundary operators by
\[
	B=\begin{pmatrix} B_{11} & B_{12}\\ B_{21} & B_{22} \\ B_{31} & B_{32} \end{pmatrix},
\]
where $B_{ik}:=b_{ik}D^{\beta_{ik}}$ with
$b_{ik}\in \R$ and $\beta_{ik}=\mbox{ord} B_{ik}\in \N_0$ for $1\leq i\leq 3, 1\leq k\leq 2. $
Let $\{r_1,r_2,r_3\}$ be a set of integers and $\beta_{ik}\leq r_i+m_k$  ($ B_{ik}=0$ if $r_i+m_k<0$). Moreover,  for $\alpha \in [0,\pi]$, we define
\[
\Sigma_\alpha:=\left\{ \begin{array}{lcl}
		\{z\in \C \mid |\arg z|< \alpha\},\qquad &\mbox{if}&\quad \alpha\in (0,\pi],\\
		(0,\infty),\qquad &\mbox{if}&\quad \alpha=0,
		\end{array}
		\right.
\]
where $\arg: \C\setminus\{0\} \to (-\pi,\pi]$ is the argument of a complex number.
   The set $\Sigma_\alpha$ is the open sector in the complex plane with vertex at the origin and angle $\alpha$. Eventually, consider the boundary value problem 
   \begin{equation}\label{BVP}
   	(\lambda -A)u=f\quad \mbox{on}\; \Omega, \qquad Bu=g \quad \mbox{on}\; \partial \Omega.
   \end{equation}
  Problem \eqref{BVP} is said to be \emph{parameter elliptic in $\Sigma_\alpha$ in sense of Douglis--Nirenberg} if the following two conditions are satisfied:
   \begin{itemize}
   \item[(C1)] \emph{Parameter-ellipticity of $A$ in $\Sigma_\alpha$:} The principal symbol $a_\pi$ of $A$ satisfies
   \[
   \det (\lambda - a_\pi(x,\xi)) \neq 0 \qquad \mbox{for all} \quad x\in \Omega, \quad \xi\in \R\setminus\{0\},\quad \lambda \in \Sigma_\alpha.
   \]
   \item[(C2)] \emph{Lopatinskij--Shapiro condition:} For any $x_0\in \partial \Omega$ and $\lambda \in \Sigma_\alpha \setminus \{0\}$ the ordinary differential equation
   \[
    [\lambda -	a_\pi(x_0,i\partial_t)]u(t)=0,\quad t>0, \qquad b_\pi(i\partial_t)u(0)=g,
   \]
    has one and only one solution $u$ with $|u(t)|\to 0$ as $t \to \infty$ for any vector $g\in \R^3$. Here, $b_\pi$ is the principal symbol of $B$.
   \end{itemize}
   
   The following theorem provides resolvent estimates for $(\lambda-A)^{-1}$ in $\Sigma_\alpha$:
   
   \begin{thm}[\cite{AES}, Theorem 6.4.1] \label{T:res} Assume that the boundary value problem \eqref{BVP} is parameter elliptic in $\Sigma_\alpha$ in the sense of Douglis--Nirenberg. Then, there exists $\lambda_0>0$ such that for $\lambda \in \Sigma_\alpha$ with $|\lambda|\geq \lambda_0$, the boundary value problem \eqref{BVP} has for any $(f,g)\in (H^s(\Omega,\R^2)\times H^s(\partial\Omega,\R^3))$ a unique solution $u \in H^{s+m_1}\times H^{s+m_2}$. In addition, there exists $c>0$ such that the following estimate is satisfied:
   \[
	   \sum_{j=1}^2\left(\|u_j\|_{H^{s+m_j}}+ |\lambda|^{2(s+m_j)} \|u_j\|_{L_2}\right) \leq c \left( \sum_{i=1}^2\|f_i\|_{H^{s-l_i}} + \sum_{j=1}^3\|g_j\|_{H^{s-r_j-\frac{1}{2}}(\partial \Omega)}\right). 
   \]
   \end{thm}

Now, we are turning to the proof of the generator property of $A(\bar u)$.
 In order to obtain the required resolvent estimates, we make use of the theory of parameter elliptic Douglis--Nirenberg systems.
 Denote by $ (\bar u_\e)_{\e}\subset (C^\infty(\overline \Omega))^2$ a sequence of functions satisfying
\[
	\|\bar u_\e-\bar u\|_\infty < \e.
\]
Clearly, since the components of $\bar u(x)$ are positive for all $x\in \overline \Omega$, there exists an index $\e_0>0$ such that also the components of $\bar u_\e>0$ are positive on $\overline \Omega$ for all $\e\in (0,\e_0)$.
We study the boundary value problem 
\begin{equation}\label{eq:BVP}
	(\lambda -A_\e)u=f\quad \mbox{on}\; \Omega, \qquad Bu=g \quad \mbox{on}\; \partial \Omega,
\end{equation}
where $A_\e:=A(\bar u_\e)$ and $B$ is the boundary operator given by
\[
	B:=\begin{pmatrix}\partial_x & 0 \\
					0 & \partial_x \\
					\partial_x^3 & 0 
		\end{pmatrix}.
\]
Notice that for $l_1=l_2=0$, $m_1=4, m_2=2$ and $r_1=-3,r_2=r_3=-1$, the matrix operators $A_\e$ and $B$ satisfy $\mbox{ord}(A_\e)_{ij}\leq l_i+m_j$ and $\mbox{ord}B_{ik}\leq r_i+m_k$ ($B_{ik}=0$ if $r_i+m_k<0$). 
In order to obtain resolvent estimates for $\lambda - A_\e$ in $\Sigma_\alpha$, we show that the boundary value problem \eqref{eq:BVP} is parameter elliptic in the sense of Douglis--Nirenberg. 
The principal symbols of $A_\e$ and $B$ are given by
\[
	a_\pi(\bar u_\e,\xi)=-\begin{pmatrix}a^\e_{11}\xi^4 & a^\e_{12}\xi^2 \\
								a^\e_{21}\xi^4  & a^\e_{22}\xi^2 \end{pmatrix},
\]
with
\begin{alignat*}{2}
&a^\e_{11}:=\frac{ (h_0^\e)^3}{3},\qquad  &&a^\e_{12}=-\frac{ (h_0^\e)^2}{2}\sigma^\prime( \Gamma_0^\e),\\
&a^\e_{21}=\frac{ (h_0^\e)^2}{2} \Gamma_0^\e, \qquad &&a^\e_{22}=-( h_0^\e  \Gamma_0^\e\sigma^\prime( \Gamma_0^\e)-\mathcal{D}),
\end{alignat*}
and
\[
	b_\pi(\xi)=i\begin{pmatrix}\xi & 0 \\
						0 & \xi \\
						-\xi^3 & 0 
			\end{pmatrix},
\]
respectively. Notice that the entries $a_{11}^\e,a_{21}^\e,a_{22}^\e>0$ and $a_{12}^\e\geq 0$  since $\mathcal{D}, h_0^\e, \Gamma_0^\e>0$ and $\sigma^\prime(s)\leq 0$ for all $s\geq 0$.

\begin{lem}\label{lem:C1} Let $\e \in (0,\e_0)$. The operator $\lambda-A_\e$ satisfies the parameter-ellipticity condition (C1) in $\C\setminus(-\infty, 0)$. 
\end{lem}
\begin{proof}
 The determinant of $\lambda-a_\pi(\bar u_\e,\xi)$ is given by
\begin{align*}
\det(\lambda - a_\pi(\bar u_\e,\xi))=\lambda^2 +(a^\e_{11}\xi+a^\e_{22}\xi^2)\lambda +(a^\e_{11}a^\e_{22}-a^\e_{12}a^\e_{21})\xi^6
\end{align*}
and observe that
\begin{align}\label{eq:ob}
\begin{split}
	\det(a_\pi(\bar u_\e,\xi))&=(a^\e_{11}a^\e_{22}-a^\e_{12}a^\e_{21})\xi^6\\
	&=-\frac{1}{12}h_0^4(x_0)\Gamma_0(x_0)\sigma^\prime(\Gamma_0(x_0))\xi^6+\mathcal{D}\frac{h_0^3(x_0)}{3}\xi^6\\
	&\geq \frac{1}{4}a^\e_{11}a^\e_{22}\xi^6\geq 0.
\end{split}
\end{align}
The roots of the polynomial 
$\det(\lambda - a_\pi(\bar u_\e,\xi))$
are given by
\[
\lambda_{\pm}=-\frac{a^\e_{11}\xi^4+a^\e_{22}\xi^2}{2}\pm \sqrt{\frac{(a^\e_{11}\xi^4+a^\e_{22}\xi^2)^2}{4}-\det(a_\pi(\bar u_\e,\xi))}.
\]
Since
\begin{align*}
&\frac{(a^\e_{11}\xi^4+a^\e_{22}\xi^2)^2}{4}-\det(a_\pi(\bar u_\e,\xi))= \frac{1}{4}\left(a^\e_{11}\xi^4-a^\e_{22}\xi^2\right)^2+a^\e_{12}a^\e_{21}\xi^6\geq 0,
\end{align*}
we deduce that all roots are real.
In view of \eqref{eq:ob} and $a_{11}^\e a_{22}^\e>0$,  we can exclude positive roots for $\xi\neq 0$ and obtain that
\[
\det (\lambda - a_\pi(\bar u_\e,\xi)) \neq 0 \qquad \mbox{for all}\quad x\in \Omega, \quad \xi\in \R\setminus\{0\},\quad \lambda \in \C\setminus (-\infty, 0).
\]
\end{proof}

\begin{lem}\label{lem:C2}  Let $\e \in (0,\e_0)$. The boundary value problem \eqref{eq:BVP} satisfies the Lopatinskij--Shapiro condition (C2) in $\Sigma_{\alpha}$ for all $\alpha \in (0,\frac{\pi}{2}]$.
\end{lem}

\begin{proof}
Let $\alpha \in (0,\frac{\pi}{2}]$. Fix $x_0\in \partial \Omega$. To lighten the notation, we set
\[
 \bar a_{ij}:=	a_{ij}^\e (x_0) \qquad \mbox{for} \quad 1\leq i,j\leq 2.
\]
We have to verify that for  $\lambda\in \Sigma_\alpha \setminus\{0\}$ the system of differential equations
\[
 [\lambda -	a_\pi(\bar u_\e(x_0),i\partial_t)]u(t)=0,\quad t>0, \qquad b_\pi(i\partial_t)u(0)=g,
\]
has one and only one solution $u$ satisfying $|u(t)|\to 0$ as $t\to \infty$ for any $g\in \R^3$. Set $u=(u_1,u_2)$. Then the system of differential equations reads
\begin{align}\label{eq:sy1}
\begin{split}
\lambda u_1 + \bar{a}_{11}u_1^{(4)}-\bar{a}_{12}u_2^{(2)}&=0,\\
\lambda u_2 + \bar{a}_{21}u_1^{(4)}-\bar{a}_{22}u_2^{(2)}&=0.
\end{split}
\end{align}
Observe that
\begin{align*}
d:=\bar a_{11}\bar a_{22}-\bar a_{12}\bar a_{21} = -\sigma^\prime(\Gamma_0^\e(x_0))\frac{(h_0^\e)^4(x_0)}{12}\Gamma_0^\e(x_0)+\mathcal D\frac{(h_0^\e)^3(x_0)}{3}>0;
\end{align*}
whence the matrix
\[
	\begin{pmatrix} \bar{a}_{11} & -\bar{a}_{12}\\
					\bar{a}_{21} & -\bar{a}_{22}
	\end{pmatrix}
\]
is invertible. As a consequence, the system \eqref{eq:sy1} can be rewritten as
\begin{align}\label{eq:sy2}
\begin{split}
u_1^{(4)} &= \frac{\lambda}{d}\left[-\bar{a}_{22}u_1 + \bar{a}_{12}u_2\right],\\
u_2^{(2)}&= \frac{\lambda}{d}\left[-\bar{a}_{21}u_1 + \bar{a}_{11}u_2\right].
\end{split}
\end{align}
The first equation implies that
\begin{equation}\label{eq:u2}
	u_2=\frac{1}{\bar{a}_{12}}\left[ \frac{d}{\lambda}u_1^{(4)}+\bar{a}_{22}u_1\right].
\end{equation}
Inserting the above equation and its second derivative into the second equation of \eqref{eq:sy2} yields the 6th-order ordinary differential equation
\begin{equation}\label{eq:6th}
u_1^{(6)}-\frac{\lambda}{d}\bar{a}_{11}u_1^{(4)}+\frac{\lambda}{d}\bar{a}_{22}u_1^{(2)}-\frac{\lambda^2}{d}=0.
\end{equation}
The general solution of \eqref{eq:6th} is given by
\begin{equation}\label{eq:gen}
u_1(t)= \sum_{k=1}^6 c_k e^{\Lambda_kt},
\end{equation}
where $\{\Lambda_k\in \C \mid k=1,\ldots,6\}$ are the roots of the characteristic polynomial
\begin{equation}\label{eq:char}
\Lambda^6-\frac{\lambda}{d}\bar{a}_{11}\Lambda^4+\frac{\lambda}{d}\bar{a}_{22}\Lambda^2-\frac{\lambda^2}{d}=0.
\end{equation}
We claim that the above polynomial has exactly three roots with strictly negative real part (say $\Lambda_k$ for $k=1,2,3$) and three roots with strictly positive real part (say $\Lambda_k$ for $k=4,5,6$). Due to the requirement that the solution tends to zero at infinity, we obtain that $c_1, c_2,$ and $c_3$ are the only nonzero constants in \eqref{eq:gen}. Let the initial condition $B_0u(0)=g$ be given by
\[
u_1^\prime(0)=g_1,\quad u_2^\prime(0)=g_2,\quad u_1^{(3)}(0)=g_3.
\]
From \eqref{eq:u2}, we infer that $u_1^{(5)}(0)=\frac{d}{\lambda}\left(\bar{a}_{12}g_2-\bar{a}_{22}g_1\right)$. Thus, any solution $u_1$ tending to zero at infinity is uniquely determined by the three initial conditions $u_1^\prime (0)$, $u_1^{(3)}(0)$, and $u_1^{(5)}(0)$.

\medskip

We are left to prove that \eqref{eq:char} has exactly three roots with strictly positive and three roots with strictly negative real part. Set $z:= \Lambda^2$, then $z$ solves the third-order equation
\begin{equation}\label{eq:z}
p(z):=dz^3 -\lambda\bar{a}_{11}z^2+\lambda\bar{a}_{22}z -\lambda^2=0.
\end{equation}
Observe that due to $\lambda\neq 0$, a zero root can be excluded. Moreover, all roots of \eqref{eq:char} appear in pairs with complex angle of $\pi$. Hence, the claim that \eqref{eq:char} has exactly three roots with strictly positive and three roots with strictly negative real part, can be proved by verifying that $p$ has no  negative real roots.
In what follows, we show that for any $\lambda \in \Sigma_\alpha \setminus\{0\}$ the above polynomial $p$ has a  no  negative root $z\in \R$.  We distinguish the cases where $\lambda$ is real and complex. If $\lambda \in \R$ then, by Descartes's Rule, the number of  negative roots of \eqref{eq:z} is bounded from above by the number of sign changes of $p(-z)$, which is zero. Hence, there do not exist any real  negative roots for $\lambda \in \R$. Concerning the case where $\lambda \in \Sigma_\alpha\setminus\{0\}$ is complex, we can write $\lambda =a+ib$, where $a,b\in \R$ with $a\geq 0$ and $b\neq 0$. Then \eqref{eq:z} decomposes to
\begin{align*}
\begin{split}
dz^3 - a\bar{a}_{11}z^2+ a\bar{a}_{22}z - (a^2-b^2)&=0,\\
-\bar{a}_{11}z^2+\bar{a}_{22}z-2a&=0.
\end{split}
\end{align*}
The second equation has the solution
\[
	z_{\pm}=\frac{\bar a_{22}}{2 \bar a_{11}}\pm \sqrt{\left(\frac{\bar a_{22}}{2 \bar a_{11}}\right)^2-\frac{2a}{\bar a_{11}}}.
\]
Notice that in view of $a\geq 0$ and $\lambda \neq 0$ a real solution $z$ satisfies $z> 0$. 
We conclude that there exist no negative real roots of $p$ for any $\lambda \in \Sigma_\alpha \setminus\{0\}$. Thus the six roots of \eqref{eq:char} are given by
\begin{align*}
\Lambda_1= |\Lambda_1|e^{i\frac{\theta_1}{2}},\qquad &\Lambda_2= |\Lambda_2|e^{i\frac{\theta_1}{2}+i\pi},\\
\Lambda_3= |\Lambda_3|e^{i\frac{\theta_2}{2}},\qquad &\Lambda_4= |\Lambda_4|e^{i\frac{\theta_2}{2}+i\pi},\\
\Lambda_5= |\Lambda_5|e^{i\frac{\theta_3}{2}},\qquad &\Lambda_6= |\Lambda_6|e^{i\frac{\theta_3}{2}+i\pi},
\end{align*}
with $\theta_i=\mbox{arg}(z_i)$, $i=1,2,3$, where $(z_i)_{1\leq i \leq 3}$ are the roots of \eqref{eq:z}. Since $p$ has no negative real roots, we have that $\theta_i \neq \pi$ for $i=1,2,3$. Hence there exist exactly three roots of \eqref{eq:char} with strictly positive and three roots with strictly negative real part, which proves the assertion.

\end{proof}

\begin{satz}\label{thm:R} For every $\e\in (0,\e_0)$, there exists $\gamma>0$ and $M\geq 1$ such that 
\begin{equation*}\
[\mbox{\emph{Re}}\lambda \geq \gamma] \subset \rho(A_\e) \quad \mbox{and}\quad 
\|(\lambda - A_\e)^{-1}\|_{\mathcal{L}(L_2\times L_2)}\leq \frac{M}{1+|\lambda|}, \quad \mbox{for all} \quad \lambda \in  [\mbox{\emph{Re}}\lambda \geq \gamma];
\end{equation*}
that is $A_\e$ is the generator of an analytic semigroup on $L_2\times L_2$.
\end{satz}

\begin{proof}
Lemma \ref{lem:C1} and Lemma \ref{lem:C2} imply that for $\alpha \in (0,\frac{\pi}{2}]$, the boundary value problem \eqref{BVP} is parameter elliptic in $\Sigma_\alpha$ in the sense of Douglis--Nirenberg. Recalling that $m_1=4,m_2=2$, $l_1=l_2=0$ and $r_i\leq -1$ for all $i=1,2,3$, the statement follows from Theorem \ref{T:res} for $s=0$.
\end{proof}

By a perturbation argument for analytic semigroups (cf. e.g. \cite[Theorem 3.2.1]{Pazy}), the existence of a maximal solution of \eqref{eq:equation} is a consequence of Proposition \ref{thm:R} and the fact that there exists a constant $c>0$ independent of $\e$ such that
\[
	\|A_\e-A(\bar u)\|_{\mathcal{L}(H^4_B\times H^2_B,L_2\times L_2)}< c\e \qquad \mbox {for all} \quad \e \in (0,\e_0).
	\]
	
\subsubsection{Uniqueness} In order to prove the uniqueness of our maximal strong solution with respect to the initial datum $u_0\in \mathcal O_\alpha$ we show that any strong solution
 \[
 \tilde u\in C^1((0,\tilde T),L_2\times L_2)\cap C((0,\tilde T),H^4_B\times H^2_B)\cap C([0,\tilde T), \mathcal O_\alpha)\qquad \tilde T\in (0,\infty]
 \]
 satisfies
 \[
 \tilde u \in  C^\eta([0,T],  H^{4\beta}_B\times H^{2\beta}_B) \quad \mbox{for all}\quad T\in (0,\tilde T),
 \]
 for some $\eta \in (0,1)$. 
Let $\tilde T\in (0,\infty]$ and $\tilde u\in C^1((0,\tilde T],E_0)\cap C((0,\tilde T],E_1)\cap C([0,\tilde T), \mathcal O_\alpha)$ be a strong solution. In particular, we have that $\tilde u\in C([0,\tilde T),H^3\times H^1)$, which allows us to deduce that $\partial_t \tilde u \in BC((0,\tilde T),(H^1)^\prime \times (H^1)^\prime)$, whence $ \tilde u \in BC^1((0,\tilde T),(H^1)^\prime \times (H^1)^\prime)$.
 If $T\in (0,\tilde T)$, the mean value theorem guarantees that there exists a constant $c=c(T)>0$ such that
 \begin{equation}\label{eq:MV}
 \|\tilde u(t_1)-\tilde u(t_0)\|_{(H^1)^\prime \times (H^1)^\prime}\leq c|t_1-t_0|\qquad \mbox{for any}\quad t_0,t_1\in [0,T].
 \end{equation}
 Using the interpolation inequality, we conclude that there exists a constant $C=C(T)>0$ such that
 \begin{align*}
	 \|\tilde u(t_1)- \tilde u(t_2)\|_{H^{4\beta}\times H^{2\beta}}&\leq \|\tilde u(t_1)- \tilde u(t_2)\|_{(H^1)^\prime \times (H^1)^\prime}^{4\frac{\alpha-\beta}{1+4\alpha}}\|\tilde u(t_1)- \tilde u(t_2)\|_{\mathcal O_\alpha}^{\frac{1+4\beta}{1+4\alpha}}\leq C |t_1-t_0|^{4\frac{\alpha-\beta}{1+4\alpha}}
 \end{align*}
for any $t_0,t_1\in [0,T]$, where we used the regularity of $\tilde u$ and \eqref{eq:MV}. Setting $\eta:=4\frac{\alpha-\beta}{1+4\alpha}\in (0,1)$, the uniqueness claim follows from Theorem \ref{T2} ii).

\bigskip

\section{Asymptotic stability}
\label{S:A}

In this section, we study the asymptotic stability of equilibria for system \eqref{eq:system}.  If $f\in L^1(\Omega)$, let us denote by
\[
\langle f \rangle:=\frac{1}{|\Omega|}\int_{\Omega} f(x)\,dx
\]
the total mass of $f$ over $\Omega$. It follows immediately from the structure of the equations in \eqref{eq:system} and the no-flux boundary conditions \eqref{eq:boundary} that any strong solution is mass conserving.

\begin{lem}[Conservation of mass]
	\label{lem:con}
	Let $u=(h,\Gamma)$ be a solution as found in Theorem \ref{T1} corresponding to the initial datum $u_0=(h_0,\Gamma_0)$. Then,
	\[
	\langle h(t)\rangle = \langle h_0\rangle,\qquad \langle \Gamma(t)\rangle =\langle \Gamma_0\rangle
	\]
	for all $t\in [0,T(u_0))$.
\end{lem}

While it is clear form \eqref{eq:system}, that any pair of constants $(h_*,\Gamma_*)$ 
is an equilibrium solution of \eqref{eq:system}--\eqref{eq:boundary}, the following discussion implies conversely if $u_*=(h_*,\Gamma_*)$ is an equilibrium of \eqref{eq:system}--\eqref{eq:boundary}, then $h_*,\Gamma_*$ are  constant.
In order to determine the set of equilibrium solutions of \eqref{eq:system}--\eqref{eq:boundary}, observe that the  energy functional
\[
E(h,\Gamma)(t):=\frac{1}{2}\int_{\Omega}|\partial_x h|^2(t,x)+\Phi(\Gamma)(t,x)\, dx
\]
(formally) decreases along solutions,
where $\Phi\in C^2(\R)$ is such that 
\[
\Phi(s)>0\qquad  \Phi^{\prime\prime}(s)=-\frac{\sigma^\prime(s)}{s}\qquad\mbox{for}\quad s>0.
\] 
 In particular (cf. \cite{EHLW4th, GW06}),
\begin{align}\label{eq:TEF}
\begin{split}
	\frac{d}{dt}E(h,\Gamma)&=-\frac{3}{2}\int_{\Omega} \left(\frac{h^{\frac{3}{2}}}{3}\partial_x^3 h+\frac{h^\frac{1}{2}}{2}\partial_x \sigma(\Gamma)\right)^2\,dx - \frac{1}{2}\int_{\Omega} \left(\frac{h^\frac{3}{2}}{2}\partial_x^3 h +h^\frac{1}{2}\partial_x \sigma(\Gamma)\right)^2\,dx\\
	&-\frac{1}{24}\int_{\Omega} h^3(\partial_x^3 h)^2\,dx-\frac{1}{8}\int_{\Omega} h\left(\partial_x \sigma(\Gamma)\right)^2\,dx-\mathcal D \int_{\Omega} \Phi^{\prime\prime}(\Gamma)\left(\partial_x\Gamma\right)^2\,dx.
\end{split}
\end{align}

Note that the regularity of our solution provided by Theorem \ref{T1} is a priori not sufficient to justify the derivative above. However, in view of the parabolic character of equation \eqref{eq:equation}, the regularity of a strong solution can be improved as follows:
 
 \begin{lem}\label{lem:imp}
 	Let
 	 \[
 	  u\in C^1((0, T),L_2\times L_2)\cap C((0, T),H^4_B\times H^2_B)\cap C([0, T), \mathcal O_\alpha),\qquad  T\in (0,\infty],
 	 \]
 	 be a solution of \eqref{eq:equation} corresponding to the initial datum $u_0=(h_0,\Gamma_0)\in \mathcal O_\alpha$, then
 	\[
 	u\in  C^\frac{5}{4}((\e,T),L_2\times L_2)\cap C^\frac{1}{4}((\e,T), H^4_B\times H^2_B)\quad \mbox{for any}\quad \e \in (0,T).
 	\]
 \end{lem}

 \begin{proof}
 	Let  \[
 	 	  u\in C^1((0, T),L_2\times L_2)\cap C((0, T),H^4_B\times H^2_B)\cap C([0, T), \mathcal O_\alpha),\qquad  T\in (0,\infty],
 	 	 \]
 	be a solution of \eqref{eq:equation} and $\e \in (0,T)$. 
 	By interpolation inequality, we deduce that 
 	\[
 	u\in C^{1-\theta}([\e,T),H^{4\theta}_B\times H^{2\theta}_B)\qquad\mbox{for any }\quad \theta \in (0,1)\setminus\{3/8,7/8\}.
 	\]
 	Choose $\theta=\frac{3}{4}$. Then clearly the coefficients of $A(u)$ are continuous in view of Sobolev embedding and
 	\begin{equation}\label{eq:A}
 	A(u)\in C^{\frac{1}{4}}([\e,T), \mathcal{H}(H^4_B\times H^{2}_B,L_2\times L_2)).
 	\end{equation}
 If $\theta=\frac{3}{4}$, then $u(t)=(h,\Gamma)(t)\in H^{3}_B\times H^{\frac{3}{2}}_B$ for all $t\in [\e,T)$. In particular, $h(t)\in C^1(\overline\Omega,\R)$, $\Gamma(t)\in C(\overline \Omega,\R)$ and $\partial_x\Gamma(t)\in H^{\frac{1}{2}}_B\subset L_4(\Omega,\R)$ by Sobolev embedding, cf. \cite[Theorem 4.57]{Demengel}. Together with our assumption that $\sigma\in C^2(\R)$, we conclude that
 	\begin{equation}\label{eq:F}
 F(u)\in C^{\frac{1}{4}}( [\e,T),L_2\times L_2).
 	\end{equation}
 Note that $w:=u$ solves the linear parabolic problem
 	\[
 		\partial_t w -A(u)w=F(u)\qquad w(\e)=u(\e).
 	\]
 	Taking into account \eqref{eq:A} and \eqref{eq:F}, we benefit from the regularizing effects of parabolic equations and obtain that the unique solution $w=u$  enjoys the following regularity, cf. \cite[Theorem II.1.2.1]{Amann95}:
 	\[
 	u\in 	u\in  C^\frac{5}{4}((\e,T),L_2\times L_2)\cap C^\frac{1}{4}((\e,T), H^4_B\times H^2_B).
 	\]
 	
 \end{proof}

The above Lemma guarantees that any strong solution $u=(h,\Gamma)$ as found in Theorem \ref{T1}, corresponding to the initial datum $u_0=(h_0,\Gamma_0)\in \mathcal O_\alpha$, satisfies 
	\[
		u\in  C^\frac{5}{4}((\e,T),L_2\times L_2)\cap C^\frac{1}{4}((\e,T), H^4_B\times H^2_B),
	\]
	where $\e \in (0,T(u_0))$ is arbitrary.
	By \cite[Theorem 1.1.5]{Lun}, the above implies that
	\[
	u\in C^{\frac{5}{4}-\alpha}((\e,T(u_0)),H^{4\alpha}_B\times H^{2\alpha}_B)\qquad\alpha\in (0,1).
	\]
	For $\alpha=\frac{1}{4}$ it follows that

	\[
	u\in C^1((\e,T(u_0)), H^1_B\times H^{\frac{1}{2}}_B)
	\]
	for any $\e \in (0,T(u_0))$.
Thereby, the time differentiation in the energy functional \eqref{eq:TEF} is well-defined for any solution $u$ provided by Theorem \ref{T1} and $t\in (0,T(u_0))$. Since all terms on the right hand side of \eqref{eq:TEF} are nonpositive, we deduce that any equilibrium solution $u_*$ of \eqref{eq:system} necessarily takes the form $u_*=(h_*,\Gamma_*)$ with $h_*,\Gamma_*>0$ constant. 
 We apply a recently established theorem on linearized stability for quasilinear equations in interpolation space \cite[Theorem 1.3]{MW} and prove the following asymptotic stability result for \eqref{eq:system}:

\begin{thm}[Asymptotic stability]
	\label{TAS}
	Let $\beta\in (\frac{7}{8},1]$ and $u_*:=(h_*,\Gamma_*)$ with $h_*,\Gamma_*>0$ being constant. Then, there exists $\e=\e(h_*)>0$, $\omega=\omega(h_*)>0$ and $M=M(h_*)\geq 1$ such that for $0<\Gamma_*<\e$ and any initial datum $u_0=(h_0,\Gamma_0)\in H^{4\beta}\times H^{2\beta}$ with
	\[
	\langle h_0\rangle =h_*,\qquad \langle \Gamma_0\rangle =\Gamma_*,
	\]
	satisfying $\|u_0-u_*\|_{H^{4\beta}\times H^{2\beta}}<\e$, the unique solution $(h,\Gamma)$ to \eqref{eq:system} exists globally in time and  
	\[
	\|u(t,\cdot)-u_*\|_{H^{4\beta}\times H^{2\beta}}\leq Me^{-\omega t}\|u_0-u_*\|_{H^{4\beta}\times H^{2\beta}}.
	\]
\end{thm}

\begin{remark} The above theorem provides an asymptotic stability result of the flat equilibrium in the interpolation space $H^{4\beta}\times H^{2\beta}$ for $\beta\in (\frac{7}{8},1]$. It thereby improves the existing result in \cite{B1}, which only states asymptotic stability for initial data in $H^4\times H^2$. We would like to mention that in \cite{BG2} an asymptotic stability result for systems including \eqref{eq:system}  is proved for initial data, which even permit high oscillations -- the price to pay are more specific size restrictions and solutions in lower regularity spaces.
\end{remark}

Before proving Theorem \ref{TAS}, let us introduce the operator
\[
Pu:=(h-\langle h \rangle, \Gamma - \langle  \Gamma \rangle)\qquad u=(h,\Gamma)\in L_2\times L_2.
\]
Then, $P$ is the projection operator from $L_2\times L_2$ onto its subset of zero mean functions. The spaces $L_2\times L_2$ and $H^4_B\times H^2_B$ decompose into
\[
L_2\times L_2 = P(L_2\times L_2)\oplus (1-P)(L_2\times L_2)\quad\mbox{and}\quad H^4_B\times H^2_B= P(H^4_B\times H^2_B)\oplus (1-P)H^4_B\times H^2_B.
\]

\begin{proof}[Proof of Theorem \ref{TAS}]
	Notice that whenever $u$ as found in Theorem \ref{T1} is a solution of
	\begin{equation}\label{eq:EQr}
	\partial_t u -A(u)u=F(u),\qquad u(0)=u_0,
	\end{equation}
	then $z:= u-u_*$, where $u_*=\langle u_0\rangle$, is a solution of
	\begin{equation}\label{eq:N}
	\partial_t z-A_*(z)z=F(z),\qquad z(0)=u_0-u_*.
	\end{equation}
	Here $A_*(z):=A(z+u_*)$.
	This is due to the conservation of mass in Lemma \ref{lem:con}. Thereby, studying the stability of the equilibrium $u_*$ for \eqref{eq:EQr} is equivalent to the stability of the zero solution $z_*=0$ of \eqref{eq:N}. We are going to apply \cite[Theorem 1.3]{MW} to prove that $z_*=0$ is an asymptotically stable solution of \eqref{eq:N}. To this end we need to verify that
\begin{equation*}
F:P\mathcal O_\beta \to P(L_2\times L_2)\quad \mbox{and}\quad A(\cdot)z_*:P\mathcal O_\beta \to P(L_2\times L_2)\quad \mbox{are Fr\'echet differentiable at}\; z_*=0
\end{equation*}
and in addition
\[
\mathbb A:=A_*(z_*)+(DA_*(z_*)[\cdot])z_*+DF(z_*),\qquad z_*=0,
\]
has a negative spectral bound; that is
$
\sup\{\re \lambda \mid \lambda \in \sigma(\mathbb A)\}<0.
$
The functions $F$ and $A_*(\cdot)z_*$ are Fr\'echet differentiable and for $z_*=0$ we have
\begin{align*}
DF(z_*)=0\qquad\mbox{and}\qquad (DA_*(z_*)[v])z_*=0\quad\mbox{for any}\; v\in P\mathcal O_\beta.
\end{align*}
Thus, we are left to show that $\mathbb A = A(u_*)$ has a negative spectral bound as an operator on $P(L_2\times L_2)$ with domain $P(H^4_B\times H^2_B)$\footnote{Observe that $(1-P)A(u_*)v=0$ for any $v\in P(H^4_B\times H^2_B)$. Thus, $A(u_*)$ is a well-defined operator from $P(H^4_B\times H^2_B)$ to $P(L_2\times L_2)$.}. This can be done in a similar way as in \cite{B1, EHLW1}. Nevertheless, we include the proof for the sake of completeness.  Recalling that 
$
\mathbb A \in \mathcal{H}(H^4_B\times H^2_B,L_2\times L_2),
$
we infer from \cite[Corollary I.1.6.3]{Amann95} that 
\[
\mathbb A \in \mathcal{H}(P(H^4_B\times H^2_B),P(L_2\times L_2)),
\]
Thus, the linear evolution equation
\begin{equation}\label{eq:z}
z_t-\mathbb Az=0\quad t>0,\qquad z(0)=z_0,
\end{equation}
admits for any initial datum $z_0\in P(L_2\times L_2)$ a unique solution
\[
z\in C((0,\infty), P(H^4_B\times H^2_B))\cap C^1((0,\infty),P(L_2\times L_2))\cap C([0,\infty), P(L_2\times L_2)).
\]
In order to prove the negative spectral bound, we are going to show that for any initial datum $z_0\in P(L_2\times L_2)$ the corresponding solution $z$ of \eqref{eq:z} satisfies
\[
\|z(t)\|_{L_2\times L_2}\leq Me^{-\omega_0t} \|z_0\|_{L_2\times L_2}\qquad \mbox{for some}\quad M\geq 1,\;\omega_0>0.
\]
 To this end, let $z_0\in P(L_2\times L_2)$ and $z=(z_1,z_2)$ be the corresponding solution of the evolution equation \eqref{eq:z}.  Recall that $u_*=(h_*,\Gamma_*)$ with $h_*,\Gamma_*>0$ constant and $\mathbb A = A(u_*)$. Thus, $z$ solves
\begin{align*}
\left\{ 
\begin{array}{lcl}
\partial_t z_1 + \frac{1}{3}h_*^3\partial_x^4z_1 + \frac{1}{2}h_*^2 \sigma^\prime(\Gamma_*)\partial_x^2 z_2 =0,\\[5pt]
\partial_t z_2 + \frac{1}{2}h_*^2\Gamma_*\partial_x^4z_1 +\left( h_*\Gamma_* \sigma^\prime(\Gamma_*)-\mathcal D\right)\partial_x^2 z_2=0.
\end{array}
\right.
\end{align*}
 Denoting by $ \hat { A}_q(h_*,\Gamma_*)$ the symmetric matrix
 \[
 \hat { A}_q(h_*,\Gamma_*):= \begin{pmatrix}
\frac{q}{3}h_*^3 & \frac{1}{4}\left(qh_*^2\sigma^\prime(\Gamma_*)+h_*^3\Gamma_*\right)\\ \frac{1}{4}\left(qh_*^2\sigma^\prime(\Gamma_*)+h_*^3\Gamma_*\right) &  h_*^2\Gamma_* \sigma^\prime(\Gamma_*)-h_*\mathcal D\\
 \end{pmatrix},\qquad q>0,
 \]
 we obtain that
\begin{align*}
\frac{1}{2}\frac{d}{dt}\left(q\| \partial_x z_1\|_{L_2}^2+h_*\| z_2\|_{L_2}^2\right)=\left\langle \hat{  A}_q(h_*\Gamma_*) \begin{pmatrix} \partial_x^3 z_1 \\ \partial_x z_2 \end{pmatrix}, \begin{pmatrix} \partial_x^3 z_1 \\ \partial_x z_2 \end{pmatrix}\right\rangle_{L_2\times L_2}.
\end{align*}
The eigenvalues of $\hat A_q (h_*,0)$ are determined by the roots of
\[
\lambda^2-\lambda \left(\frac{q}{3}+h_*\mathcal D\right)+\frac{q}{3}h_*^4\mathcal D-\frac{1}{6}q^2h_*^4|\sigma^\prime(0)|^2=0.
\]
We conclude that $\hat A_q(h_*,0)$ is positive definite if
\[
0<q<\frac{16\mathcal D}{3|\sigma^\prime(0)|}.
\]
Thus, whenever $q$ satisfies the above condition, there exists $\e=\e(h_*)>0$ such that for any $0<\Gamma_*<\e$ the matrix $A_q(h_*,\Gamma_*)$ is positive definite. In particular, there exists a constant $c_0=c_0(h_*)>0$ such that
\[
\frac{1}{2}\frac{d}{dt}\left(q\|\partial_x z_1\|_{L_2}^2+h_*\| z_2\|_{L_2}^2\right)\leq -c _0\left(\|\partial_x^3 z_1\|_{L_2}^2+\|\partial_x z_2\|_{L_2}^2\right).
\]
Taking into account that $\partial_xz_1=\partial_x z_2=\partial_x^3z_1=0$ at $\partial \Omega$, and $\partial_x^2z_1$ has zero mean over $\Omega$, applying the Poincar\'e inequality implies the existence of a constant $c_1=c_1(h_*)>0$ such that
\[
\frac{d}{dt}\left(q\|z_1\|_{L_2}^2+h_*\|z_2\|_{L_2}^2\right)\leq -c_1 \left(q\|z_1\|_{L_2}^2+h_*\|z_2\|_{L_2}^2\right).
\]
Eventually, observing that
\[
||| z|||_{L_2\times L_2}:=q\|z_1\|_{L_2}^2+h_*\|z_2\|_{L_2}^2,\qquad z\in L_2\times L_2,
\]
constitutes an equivalent norm on $L_2\times L_2$, we infer from the above inequality that there exists a constant $\omega=\omega(h_*)>0$ and $M=M(h_*)\geq 1$ such that
\[
\|z(t)\|_{L_2\times L_2}\leq Me^{-\omega_0t} \|z_0\|_{L_2\times L_2}.
\]
Thereby, the proof is finished.

\end{proof}

\bigskip

\subsection*{Acknowlegments}
Funded by the Deutsche Forschungsgemeinschaft (DFG, German Research
Foundation) -- Project-ID 258734477 -- SFB 1173. In addition, the author would like to thank Bogdan Matioc for valuable suggestions to improve the manuscript. 

\bigskip

\bibliographystyle{siam}
\bibliography{/home/gabriele/Dokumente/Bibtex/B_Thin_Film_Wellposedness.bib}

\begin{thebibliography}{10}

\bibitem{AES}
{\sc M.~S. Agranovich, Y.~V. Egorov, and M.~A. Shubin}, {\em {Partial
  differential equations IX: Elliptic boundary value problems}}, Springer,
  1997.

\bibitem{AmannNon}
{\sc H.~Amann}, {\em {Nonhomogeneous linear and quasilinear elliptic and
  parabolic boundary value problems}}, in Function spaces, differential
  operators and nonlinear analysis (Friedrichroda, 1992), vol.~133 of
  Teubner-Texte Math., Teubner, Stuttgart, 1993, pp.~9--126.

\bibitem{Amann95}
\leavevmode\vrule height 2pt depth -1.6pt width 23pt, {\em {Linear and
  quasilinear parabolic problems, volume I: Abstract linear theory}},
  Birkh{\"{a}}user Verlag, Berlin, 1995.

\bibitem{B1}
{\sc G.~Bruell}, {\em {Modeling and analysis of a two-phase thin film model
  with insoluble surfactant}}, Nonlinear Anal. Real World Appl., 27 (2016),
  pp.~124--145.

\bibitem{B2}
\leavevmode\vrule height 2pt depth -1.6pt width 23pt, {\em {Weak solutions to a
  two-phase thin film model with insoluble surfactant driven by capillary
  effects}}, Journal of Evolution Equations, 17 (2017), pp.~1341--1379.

\bibitem{BG2}
{\sc G.~Bruell and R.~Granero-Belinchon}, {\em {On a thin film model with
  insoluble surfactant}}, preprint.

\bibitem{CT13}
{\sc M.~Chugunova and R.~M. Taranets}, {\em {Nonnegative weak solutions for a
  degenerate system modeling the spreading of surfactant on thin films}}, Appl.
  Math. Res. Express. AMRX,  (2013), pp.~102--126.

\bibitem{Demengel}
{\sc F.~Demengel and G.~Demengel}, {\em {Functional spaces for the theory of
  elliptic partial differential equations}}, Universitext, Springer, London;
  EDP Sciences, Les Ulis, 2012.

\bibitem{EHLW11}
{\sc J.~Escher, M.~Hillairet, P.~Lauren{\c{c}}ot, and C.~Walker}, {\em {Global
  weak solutions for a degenerate parabolic system modeling the spreading of
  insoluble surfactant}}, Indiana Univ. Math. J., 60 (2011), pp.~1975--2019.

\bibitem{EHLW1}
\leavevmode\vrule height 2pt depth -1.6pt width 23pt, {\em {Thin film equations
  with soluble surfactant and gravity: modeling and stability of steady
  states}}, Math. Nachr., 285 (2012), pp.~210--222.

\bibitem{EHLW4th}
\leavevmode\vrule height 2pt depth -1.6pt width 23pt, {\em {Weak solutions to a
  thin film model with capillary effects and insoluble surfactant}},
  Nonlinearity, 25 (2012), pp.~2423--2441.

\bibitem{GW06}
{\sc H.~Garcke and S.~Wieland}, {\em {Surfactant spreading on thin viscous
  films: nonnegative solutions of a coupled degenerate system}}, SIAM J. Math.
  Anal., 37 (2006), pp.~2025--2048.

\bibitem{JG}
{\sc O.~E. Jensen and J.~B. Grotberg}, {\em {Insoluble surfactant spreading on
  a thin viscous film: shock evolution and film rupture}}, J. Fluid Mech., 240
  (1992), pp.~259--288.

\bibitem{Lun}
{\sc A.~Lunardi}, {\em {Analytic semigroups and optimal regularity in parabolic
  problems}}, Modern Birkh{\"{a}}user Classics, Birkh{\"{a}}user/Springer Basel
  AG, Basel, 1995.

\bibitem{MW}
{\sc B.~Matioc and C.~Walker}, {\em {On the principle of linearized stability
  in interpolation spaces for quasilinear evolution equations}},
  arXiv:1804.10523,  (2018).

\bibitem{Pazy}
{\sc A.~Pazy}, {\em {Semigroups of linear operators and applications to partial
  differential equations}}, vol.~44 of Applied Mathematical Sciences,
  Springer-Verlag, New York, 1983.

\bibitem{TriI}
{\sc H.~Triebel}, {\em {Interpolation theory, function spaces, differential
  operators}}, Johann Ambrosius Barth, Heidelberg, second~ed., 1995.

\end{thebibliography}

\end{document}